\documentclass[12pt]{amsart}
\usepackage[english]{babel}
\usepackage[utf8]{inputenc}
\usepackage[T1]{fontenc}
\usepackage[nomath]{lmodern} 

\usepackage{amssymb,amsmath,amsthm,amsfonts}
\usepackage{thmtools,mathtools}
\usepackage{microtype}
\allowdisplaybreaks

\usepackage[dvipsnames]{xcolor}
\usepackage{graphicx} 
\usepackage{tikz,tikz-cd,pgfplots}
\pgfplotsset{compat=newest}

\usepackage{caption} 
\usepackage{fmtcount} 

\usepackage[breaklinks,pdfencoding=auto,psdextra]{hyperref}
\usepackage{enumitem} 

\def\Macaulay{{\tt Macaulay2}}
\usepackage{listings}
\makeatletter
\lst@CCPutMacro
\lst@ProcessLetter{"27}\textquotesingle
\lst@ProcessOther{"2D}{-}
\@empty\z@\@empty
\makeatother
\lstdefinelanguage{Macaulay2}{
comment=[l]{--},
alsoletter={'},
alsoother={_},
}
\theoremstyle{plain}
\newtheorem{theorem}{Theorem}[section]
\newtheorem{lemma}[theorem]{Lemma}
\newtheorem{corollary}[theorem]{Corollary}
\newtheorem{proposition}[theorem]{Proposition}
\theoremstyle{definition}
\newtheorem{definition}[theorem]{Definition}

\newtheorem{remark}[theorem]{Remark}

\newcounter{chr}
\loop\ifnum\value{chr}<27
  \expandafter\edef\csname b\Alph{chr}\endcsname{\noexpand\mathbf{\Alph{chr}}}
  \expandafter\edef\csname c\Alph{chr}\endcsname{\noexpand\mathcal{\Alph{chr}}}
  \expandafter\edef\csname f\Alph{chr}\endcsname{\noexpand\mathfrak{\Alph{chr}}}
\addtocounter{chr}{1}
\repeat

\def\m{\mathfrak m}
\def\p{\mathfrak p}
\def\q{\mathfrak q}

\def\Aut{\operatorname{Aut}}

\def\Def{\operatorname{Def}}

\def\disc{\operatorname{disc}}

\def\div{\operatorname{div}}
\def\End{\operatorname{End}}
\def\Ext{\operatorname{Ext}}

\def\GL{\operatorname{GL}}
\def\Gr{\operatorname{Gr}}

\def\Id{\operatorname{Id}}

\def\O{\operatorname{O}}

\def\PSL{\operatorname{PSL}}
\def\Pic{\operatorname{Pic}}

\def\rk{\operatorname{rk}}
\def\Sing{\operatorname{Sing}}
\def\SL{\operatorname{SL}}

\def\Sym{\operatorname{Sym}}

\def\KKK{{\mathrm{K3}}}
\def\trans{{\mathrm{trans}}}

\let\ordexists\exists
\def\exists{\operatorname{\ordexists}}
\let\ordforall\forall
\def\forall{\operatorname{\ordforall}}

\def\inner#1{{\left<{#1}\right>}}
\def\set#1{{\left\{{#1}\right\}}}
\def\setmid#1#2{{\left\{{#1}\;\middle|\;{#2}\right\}}}

\def\tilde{\widetilde}
\def\setminus{\smallsetminus}

\def\git{/\!\!/}

\def\bw#1{{\mathchoice%
 {\textstyle{\bigwedge\mkern-4.5mu^{#1}\mkern1mu}}%
 {\textstyle{\bigwedge\mkern-4.5mu^{#1}\mkern1mu}}%
 {\scriptstyle{\bigwedge\mkern-5mu^{#1}}}%
 {\scriptscriptstyle{\bigwedge\mkern-5mu^{#1}}}%
}}

\def\longarrow#1#2{\mathchoice{#2}{#1}{#1}{#1}}
\def\to{\longarrow{\rightarrow}{\longrightarrow}}
\def\simto{\longarrow{\xrightarrow\sim}{\stackrel\sim\longrightarrow}}
\def\into{\longarrow{\hookrightarrow}{\lhook\joinrel\longrightarrow}}

\let\shortmapsto\mapsto
\def\mapsto{\longarrow{\shortmapsto}{\longmapsto}}

\usepackage[margin=1in]{geometry}

\def\bsigma{{\sigma_0}}
\def\Co{\operatorname{Co}}
\def\Stab{\operatorname{Stab}}
\def\KS{\operatorname{KS}}
\def\GG{\bG}
\def\Fano{X_{\GG}^{\text{Fano}}}
\def\EPW{X_{\GG}^{\text{EPW}}}
\def\DV{X_{\GG}^{\text{DV}}}
\def\MDV{\cM_{\mathrm{DV}}}
\def\MDVsm{\MDV^{\mathrm{sm}}}

\begin{document}

\title{A special Debarre--Voisin fourfold}
\author[J.~Song]{Jieao Song}
\address{Université Paris Cité, CNRS, IMJ-PRG, 75013 Paris, France}
\email{\href{mailto:jieao.song@imj-prg.fr}{\tt jieao.song@imj-prg.fr}}
\date{\today}
\begin{abstract}
Consider the finite simple group~$\GG\coloneqq\PSL(2,\bF_{11})$ of order 660,
which has an irreducible
representation~$V_{10}$ of dimension 10. In this note, we study a special
trivector $\bsigma\in \bw3V_{10}^\vee$ that is
$\GG$-invariant. Following the construction of Debarre--Voisin, we obtain a
smooth hyperkähler fourfold $X_6^\bsigma\subset\Gr(6,V_{10})$ with many
symmetries. We will also look at the associated Peskine variety
$X_1^\bsigma\subset \bP(V_{10})$, which is highly symmetric as well and admits
55 isolated singular points. It will help us to better understand the geometry of
the special Debarre--Voisin fourfold~$X_6^\bsigma$.
We also discuss an application of this example to the global geometry of
the moduli space of Debarre--Voisin fourfolds.
\end{abstract}

\maketitle

\section{Introduction}
The study of automorphism groups for K3 surfaces and higher dimensional
hyperkähler manifolds is a rich subject that has many deep relations with
lattice theory and representation theory of simple groups. For example,
in \cite{Mukai:AutK3}, Mukai showed that a finite group of symplectic automorphisms of
a K3 surface is always a subgroup of the Mathieu group $M_{23}$.
Similarly, in \cite[Theorem~7.2.4]{monthesis}, Mongardi showed that a finite group of
symplectic automorphisms of a hyperkähler manifold of $\KKK^{[2]}$-type is a
subgroup of the Conway group $\Co_1$. Moreover, for a such manifold $X$, the
maximal prime order of any symplectic automorphism is 11, and in this case, $X$
must have maximal Picard rank 21, so it is isolated in the moduli.

Consider the finite simple group~$\GG\coloneqq\PSL(2,\bF_{11})$ of order 660.
Mongardi constructed a special cubic fourfold, as well as a special
Eisenbud--Popescu--Walter sextic with a faithful $\GG$-action. From these, one
obtains two hyperkähler fourfolds of $\KKK^{[2]}$-type---the corresponding Fano
variety of lines and double EPW sextic---that are highly
symmetric (see~\cite[Section~4.5]{monthesis} and~\cite{dm}). We also note
that a complete classification of symplectic automorphism groups for cubic
fourfolds is available in~\cite{LazaZheng}.

In this paper, we study an explicit example of a hyperkähler fourfold of
$\KKK^{[2]}$-type in the Debarre--Voisin family that also admits a faithful
$\GG$-action. A key feature of this example is that we can describe explicitly
its Picard lattice using the geometry of some associated Fano varieties.

Let~$V_{10}$ be a 10-dimensional complex vector space. A {\em
Debarre--Voisin variety} $X_6^\sigma$ is defined inside the Grassmannian
$\Gr(6,V_{10})$ from the datum of a trivector $\sigma\in \bw3V_{10}^\vee$. By
studying the representations of the group~$\GG$, it is not hard to find a
candidate for the special trivector~$\bsigma$: denote by~$V_{10}$ one of the two $10$-dimensional
irreducible representations of $\GG$; there exists a unique (up to
multiplication by a scalar) trivector $\bsigma\in\bw3V_{10}^\vee$ that
is~$\GG$-invariant.

Using the general results obtained in~\cite{bs} on the geometry of
Debarre--Voisin varieties and associated Peskine varieties, one can study in
detail the geometry of this special Debarre--Voisin variety~$X_6^\bsigma$. We
prove the following results.
\begin{theorem}Let~$\bsigma\in \bw3V_{10}^\vee$ be the special $\GG$-invariant
trivector.
\begin{enumerate}
\item (Proposition~\ref{prop:smoothness}) The Debarre--Voisin variety
$X_6^\bsigma\subset\Gr(6,V_{10})$ is smooth of dimension~$4$.
\item (Proposition~\ref{prop:55-points}) The associated Peskine variety
$X_1^\bsigma\subset\bP(V_{10})$ has $55$ isolated singular points. The group
$\GG$ acts transitively on them.
\item (Proposition~\ref{prop:aut}) The group $\Aut^s_H(X_6^\bsigma)$ of symplectic
automorphisms that fix the polarization~$H$ on $X_6^\bsigma$ is isomorphic
to~$\GG$.
\item One can give an explicit description of the Picard lattice of
$X_6^\bsigma$, which has maximal rank $21$. It is spanned by $55$ $(-2)$-classes
(see \eqref{eq:picX6} for the Gram matrix). Moreover,
if we denote by $H^2_{\trans}(X_6^\bsigma)$ the transcendental lattice and by
$T\coloneqq H^2(X_6^\bsigma,\bZ)^\GG$ the $\GG$-invariant sublattice,
we have the following isomorphisms of lattices (Proposition~\ref{prop:lattices})
\begin{gather*}
H^2_{\trans}(X_6^\bsigma)\simeq L_{11}\coloneqq\begin{psmallmatrix}2 & 1 \\1 & 6 \end{psmallmatrix},
\qquad T=H^2_{\trans}(X_6^\bsigma)\oplus\inner{H}\simeq
L_{11}\oplus(22),\\
\Pic(X_6^\bsigma)\simeq
U\oplus E_8(-1)^{\oplus 2}\oplus L(-1),
\end{gather*}
where the component $L$ can be taken to be both
$\begin{psmallmatrix}2&1&0\\1&2&1\\0&1&8\end{psmallmatrix}$
and $L_{11}\oplus(2)$.
\item (Proposition~\ref{prop:uniqueness}) $X_6^\bsigma$ can be characterized as the unique Debarre--Voisin fourfold admitting a symplectic automorphism of order~$11$.
\item $X_6^\bsigma$ is birationally isomorphic to the Hilbert square of a K3 surface (Proposition~\ref{prop:S2}) and is special in the sense of Hassett for all possible discriminants $d\ge 24$ (Proposition~\ref{prop:disc}).
\end{enumerate}
\end{theorem}

The property of being Hassett special for all possible discriminants $d\ge 24$ has a nice implication on the global geometry of the moduli space of Debarre--Voisin fourfolds. Namely, we have two different moduli spaces in this setting: the GIT moduli space $\MDV$ of trivectors and the moduli space $\cM_{22}^{(2)}$ of polarized hyperkähler manifolds. The Debarre--Voisin construction provides a rational map
\[
\m\colon \MDV\dashrightarrow \cM_{22}^{(2)},
\]
which is proved to be birational \cite[Theorem~1.8]{ogrady}. Moreover, one can show that the restriction of $\m$ to the open locus $\MDVsm$ of trivectors defining a smooth Debarre--Voisin fourfold is an open immersion (Proposition~\ref{prop:m_immersion}).

When we resolve the indeterminacies of this map, the image of each exceptional divisor is called \emph{Hassett--Looijenga--Shah} (HLS) (see Definition~\ref{def:HLS}),
which reflects a difference between the GIT and the Baily--Borel compactifications.
The result on $X_6^\bsigma$ implies that all Heegner divisors $\cD_d$ for $d\ge 24$ are not HLS (Corollary~\ref{cor:HLS}). Combined with the results of \cite{DHOV} and \cite{oberdieck}, one concludes that a Heegner divisor $\cD_d$ is HLS if and only if $d\in\set{2,6,8,10,18}$. We discuss this in Section~\ref{sec:HLS}.

\subsection*{Notation}
We use $\sigma$ to denote a trivector and $\bsigma$ to denote the
special $\GG$-invariant trivector.

\subsection*{Acknowledgements}
I would like to express my sincere gratitude to my advisor, Olivier Debarre,
for his support and guidance. I thank Frédéric Han for his help on some
computations, Vladimiro Benedetti for interesting discussions, and Laurent Manivel and Giovanni
Mongardi for their interest in this work, as well as the referee for very helpful comments. I also
thank Grzegorz Kapustka for pointing to me some related works by his student,
Tomasz Wawak.

A major part of the results obtained in this note are based on computer
algebra. I would like to thank the developers of \Macaulay{} and {\tt Sage}
for making their software open-source and available to all.

\section{The special trivector}
\label{sec:trivector}
We first give the construction of the special trivector
$\bsigma\in\bw3V_{10}^\vee$.

The finite simple group $\GG\coloneqq \PSL(2,\bF_{11})$ of order $660$
admits eight different irreducible complex representations: two of them are of
dimension~$5$ and will be denoted by~$V_5$ and~$V_5^\vee$. They are the dual
to each other.

A classical result is that the symmetric
power~$\Sym^3V_5^\vee$---the space of cubic polynomials on~$V_5$---admits an
irreducible subrepresentation of dimension~$1$: for a suitable choice of basis
$(y_0,\dots,y_4)$ of~$V_5^\vee$, this corresponds to the Klein cubic with
equation
\[y_0^2y_1+y_1^2y_2+y_2^2y_3+y_3^2y_4+y_4^2y_0\in\Sym^3 V_5^\vee.\]
In~\cite{adl}, Adler showed that the automorphism group of this smooth cubic is
precisely the group~$\GG$.

The wedge product~$\bw2V_5$ gives another irreducible representation, of
dimension~$10$, which is self-dual and will be denoted by~$V_{10}$.  We consider
elements of~$\bw3V_{10}^\vee$.  A computation of
characters tells us that this representation of~$\GG$ also admits one
irreducible subrepresentation of dimension~$1$, generated by a
$\GG$-invariant trivector $\bsigma$. The characters of all eight irreducible
representations of~$\GG$ as well as the character of $\bw3V_{10}^\vee$ can be
found in Section~\ref{sec:appendix}, Table~\ref{tbl:PSL-char}. Note that the
other irreducible representation $V_{10}'$ of dimension~10 does not provide
$\GG$-invariant trivectors (see also Remark~\ref{rmk:end} on the uniqueness of the trivector $\bsigma$).

We now give a concrete description of the special trivector $\bsigma$ in terms
of coordinates in a suitable basis.
The subgroup~$\bB$ of~$\GG$ of upper triangular matrices can be
generated by the elements
\[
P=\begin{pmatrix}1&1\\0&1\end{pmatrix}\quad\text{and}\quad
R=\begin{pmatrix}4&0\\0&3\end{pmatrix},
\]
of the respective orders 11 and 5.
Write~$\zeta=e^{2\pi i/11}$ and $\rho\colon
\GG\to \GL(V_5^\vee)$ for the representation~$V_5^\vee$.  In a suitable
basis $(y_0,\dots,y_4)$ of~$V_5^\vee$, the matrices of~$P$ and~$R$ are
\begin{equation}\label{eq:PR}
\rho(P)=\begin{pmatrix}
\zeta^1&0&0&0&0\\
0&\zeta^9&0&0&0\\
0&0&\zeta^4&0&0\\
0&0&0&\zeta^3&0\\
0&0&0&0&\zeta^5\end{pmatrix}\quad\text{and}\quad
\rho(R)=\begin{pmatrix}
0&0&0&0&1\\
1&0&0&0&0\\
0&1&0&0&0\\
0&0&1&0&0\\
0&0&0&1&0\end{pmatrix}.
\end{equation}
Note that one can already identify the equation of the~$\GG$-invariant Klein
cubic using only these two elements, instead of the whole group~$\GG$.

The elements~$y_{ij}\coloneqq y_i\wedge y_j$ form a basis of~$V_{10}^\vee$.
In this basis, we see that~$P$ acts diagonally and~$R$ as a permutation
(see Table~\ref{tbl:PR}; note that we have chosen a particular order in which the
action of $R$ is very simple).
\begin{table}[ht]
\def\arraystretch{1.5}
$\begin{array}{|r|cccccccccc|}
\hline
&y_{01}&y_{12}&y_{23}&y_{34}&y_{40}&y_{02}&y_{13}&y_{24}&y_{30}&y_{41}\\
\hline
\text{Eigenvalues of }\bw2\rho(P)&\zeta^{10}&\zeta^2&\zeta^7&\zeta^8&\zeta^6&\zeta^5&\zeta^1&\zeta^9&\zeta^4&\zeta^3\\
\text{Action of }\bw2\rho(R)&y_{12}&y_{23}&y_{34}&y_{40}&y_{01}&y_{13}&y_{24}&y_{30}&y_{41}&y_{02}\\
\hline
\end{array}$
\caption{\label{tbl:PR}The action of~$P$ and~$R$ in the basis~$(y_{ij})$}
\end{table}
We may easily verify that the
space of trivectors invariant under the action of~$P$ and~$R$ is of dimension 2
and is spanned by the $\bB$-invariant trivectors
\begin{gather*}
\sigma_1\coloneqq y_{01}\wedge y_{23}\wedge y_{02}+
y_{12}\wedge y_{34}\wedge y_{13}+
y_{23}\wedge y_{40}\wedge y_{24}+
y_{34}\wedge y_{01}\wedge y_{30}+
y_{40}\wedge y_{12}\wedge y_{41},\\
\sigma_2\coloneqq y_{01}\wedge y_{41}\wedge y_{24}+
y_{12}\wedge y_{02}\wedge y_{30}+
y_{23}\wedge y_{13}\wedge y_{41}+
y_{34}\wedge y_{24}\wedge y_{02}+
y_{40}\wedge y_{30}\wedge y_{13}.
\end{gather*}
To identify the unique~$\GG$-invariant trivector, we must look at some elements
in~$\GG\setminus \bB$.  Since the explicit description for the
representation~$V_5$ is known \cite{atlas}, we will pick one such element and
compute its matrix explicitly.

The group~$\GG$ admits a presentation with two generators
$a,b$ and relations $a^2=b^3=(ab)^{11}=[a,babab]^2=1$. We can take
$a=\begin{psmallmatrix}0&-1\\1&0\end{psmallmatrix}$ and
$b=\begin{psmallmatrix} 0&1\\-1&-1\end{psmallmatrix}$.
One may check that $ab=P$ while $bbabababbabababb=R$.
Matrices for~$\rho(a)$ and~$\rho(b)$ are provided by~\cite{atlas}, so the
representation is completely determined.
Choose a suitable basis of~$V_5^\vee$ consisting of eigenvectors of
$\rho(ab)=\rho(P)$. In this basis,
the matrices of $P$ and $R$ are as in~\eqref{eq:PR}. Since the element~$a$
does not lie in the subgroup~$\bB$, we use its matrix in this new basis to
verify that the unique (up to multiplication by a scalar) $\GG$-invariant
trivector is~$\bsigma\coloneqq\sigma_1+\sigma_2$.

From now on, we will rewrite the basis~$(y_{ij})$ as $(x_0,\dots,x_9)$ in the
order chosen in Table~\ref{tbl:PR}, so the trivector $\bsigma$ is given by
\begin{gather*}
\bsigma=x_0\wedge x_2\wedge x_5+
x_1\wedge x_3\wedge x_6+
x_2\wedge x_4\wedge x_7+
x_3\wedge x_0\wedge x_8+
x_4\wedge x_1\wedge x_9\\+
x_0\wedge x_9\wedge x_7+
x_1\wedge x_5\wedge x_8+
x_2\wedge x_6\wedge x_9+
x_3\wedge x_7\wedge x_5+
x_4\wedge x_8\wedge x_6,
\end{gather*}
or more succinctly,
\begin{equation}\label{eq:sigma0}
\bsigma=[025]+[136]+[247]+[308]+[419]+[097]+[158]+[269]+[375]+[486].
\end{equation}
We have, therefore, shown the following result.
\begin{proposition}
Up to multiplication by a scalar, the trivector~$\bsigma$ in~\eqref{eq:sigma0}
is the unique~$\GG$-invariant trivector in~$\bw3V_{10}^\vee$, where~$V_{10}$ is
the $10$-dimensional irreducible $\GG$-representation given in
Table~\ref{tbl:PSL-char}.
\end{proposition}

\section{The Debarre--Voisin fourfold}
The Debarre--Voisin variety associated with a non-zero trivector $\sigma$ is
the scheme
\[
X^\sigma_6\coloneqq\setmid{[V_6]\in\Gr(6,V_{10})}{\sigma|_{V_6}=0}
\]
in the Grassmannian $\Gr(6,V_{10})$ parametrizing those $[V_6]$
on which $\sigma$ vanishes. Its expected dimension is~4.
For $\sigma$ general, it is shown in \cite{dv} that $X_6^\sigma$ is a smooth
hyperkähler fourfold of $\KKK^{[2]}$-type. The Plücker polarization
on $\Gr(6,V_{10})$ induces a polarization $H$ on $X_6^\sigma$, which is
primitive and of Beauville--Bogomolov--Fujiki square $22$ and divisibility $2$.

The variety
\[
X^\sigma_3\coloneqq\setmid{[V_3]\in\Gr(3,V_{10})}{\sigma|_{V_3}=0}
\]
is the Plücker hyperplane section associated with $\sigma$. It has dimension~20.
We have the following criterion for the smoothness of $X^\sigma_3$ and
$X^\sigma_6$
\cite[Lemma~2.1]{bs}.
\begin{lemma}
\label{lemma:smoothness-X3-X6}
The Debarre--Voisin $X^\sigma_6$ is not smooth of dimension $4$ at a point
$[V_6]$ if and only if there exists $V_3\subset V_6$ satisfying the
degeneracy condition
\[
\sigma(V_3,V_3,V_{10})=0.
\]
In particular, $X^\sigma_6$ is smooth of dimension~$4$ if and only if
$X^\sigma_3$ is smooth.
\end{lemma}

In the case of $\bsigma$, the smoothness of $X_3^\bsigma$ can be verified
directly using computer algebra (thanks are due to Frédéric Han for his help with this
computation).
\begin{proposition}
\label{prop:smoothness}
For the special trivector $\bsigma$ defined in~\eqref{eq:sigma0},
the hyperplane section $X_3^\bsigma$ and, hence, the special Debarre--Voisin
$X_6^\bsigma$ are both smooth.
\end{proposition}
\begin{proof}
A direct check of the smoothness of $X_3^\bsigma$ using its ideal is not
feasible, since there are too many variables and equations. Instead, we
can check the smoothness on each chart of $\Gr(3,V_{10})$ where it is defined
by one single cubic polynomial in an affine space $\bA^{21}$, using the
Jacobian criterion. See Section~\ref{m2:smoothness} for the \Macaulay{} code.
\end{proof}

The action of $\GG$ on $V_{10}$ induces an action on $X_6^\bsigma$
that preserves the polarization~$H$ and, hence, a homomorphism of groups
$\GG\to\Aut_H(X_6^\bsigma)$. The induced action of the automorphism group on
the symplectic form gives a character $\chi\colon\Aut_H(X_6^\bsigma)\to \bC^*$,
and we obtain the following short exact sequence
\[
1\to \Aut^s_H\to \Aut_H\stackrel{\chi}\to
\overline{\Aut_H}\to 0,
\]
where $\Aut^s_H=\Aut^s_H(X_6^\bsigma)=\ker\chi$ is the subgroup of symplectic
automorphisms, and $\overline{\Aut_H}=\overline{\Aut_H}(X_6^\bsigma)$ is the
image of $\chi$, which is a finite subgroup of $\bC^*$ and is abelian. Since the
group $\GG$ is simple and non-abelian, we may deduce that the image of the
homomorphism $\GG\to \Aut_H(X_6^\bsigma)$ must be contained in the subgroup
$\Aut^s_H(X_6^\bsigma)$
of symplectic automorphisms. We shall see that this is an isomorphism onto
$\Aut^s_H(X_6^\bsigma)$.

\section{The Peskine variety}
With a trivector $\sigma$, we can associate yet another variety: the Peskine
variety
\[
X^\sigma_1\coloneqq\setmid{[V_1]\in\bP(V_{10})}{\rk\sigma|_{V_1}\le 6}.
\]
More precisely, for each $[V_1]\in\bP(V_{10})$, the skew-symmetric $2$-form
$\sigma(V_1,-,-)$ generically has rank~$8$, and the Peskine variety
$X^\sigma_1$ is the locus where this rank drops
to $6$ or less.  Equivalently, given a basis $(e_i)$ of $V_{10}$, we can
identify $\sigma$ with a $10\times 10$ skew-symmetric matrix with entries
$f_{ij}\coloneqq\sigma(e_i,e_j,-)$. Then $X_1^\sigma$ is defined in
$\bP(V_{10})$ by all the $8\times 8$-Pfaffians of this matrix.
It has expected dimension $6$ and degree $15$. Its smoothness is characterized
by the following lemma \cite[Lemma~2.8]{bs}.
\begin{lemma}
If the Peskine variety $X^\sigma_1$ is not smooth of dimension $6$ at a point
$[V_1]$, then either $\rk\sigma|_{V_1}\le 4$, or there exists a $V_3\supset V_1$
such that $\sigma(V_3,V_3,V_{10})=0$.
\end{lemma}

In the case of the special trivector $\bsigma$, since $X_3^\bsigma$ was shown
to be smooth, the second case does not happen by
Lemma~\ref{lemma:smoothness-X3-X6}.
Therefore, the singular locus of $X_1^\bsigma$ is precisely the locus where the
rank of $\bsigma$ drops to even less.
Equivalently, it is defined by all the $6\times 6$-Pfaffians of $\bsigma$
seen as a skew-symmetric matrix. This allows us to explicitly compute the ideal
of the singular subscheme using \Macaulay{} (see Section~\ref{m2:ideals}).
In particular, we may verify that the rank-$4$ locus $\Sing(X_1^\sigma)$ is a
subscheme of dimension~0 and length~55, while the rank-2 locus is empty. Also,
the rank-6 locus $X_1^\bsigma$ is, indeed, of expected dimension~6 and degree~15.
We now show that $\Sing(X_1^\sigma)$ is reduced.

\begin{proposition}
\label{prop:55-points}
For the special trivector $\bsigma$, the singular locus of the Peskine variety
$X_1^\bsigma$ consists of the $55$ distinct points
\[
(p_{i,j})_{0\le i\le 4,0\le j\le 10},
\]
where the rank of $\bsigma$ is equal to $4$ instead of $6$.  The group $\GG$
acts transitively on these $55$ points.
\end{proposition}
\begin{proof}
Since we have already obtained the ideal of the rank-$4$ locus, to verify that
there are 55 distinct points, we can compute the radical in \Macaulay{} to
check that it is, indeed, reduced.

Alternatively, we can use a Gröbner bases computation to obtain the explicit
coordinates for the underlying points and verify that there are $55$ distinct
solutions (the author wrote a \Macaulay{} package, {\tt RationalPoints2}, that
can perform this computation to produce the explicit coordinates). However, to better understand the action of
the group~$\GG$ on these points, we will explain another step-by-step procedure
to solve the system using this group action.

We first consider the hyperplane $x_0+x_1+x_2+x_3+x_4=0$. The intersection of
this hyperplane with the singular locus is a subscheme of length $5$.
To compute the coordinates of these $5$ points, we can use elimination and
obtain a degree-5 equation for $X\coloneqq x_1/x_0$
\[
1-4X+2X^2+5X^3-2X^4-X^5=0.
\]
This polynomial splits in the cyclotomic field $\bQ(\zeta)$, and all the roots
are real, so its splitting field is the real subextension of $\bQ(\zeta)/\bQ$
of degree $5$. We take one real root $\zeta^7+\zeta^6+\zeta^5+\zeta^4$, which
allows us to recover the coordinates of one point $p_{0,0}$. The action of the
order-$5$ element $R$ now recovers all the five points on the hyperplane. We
denote these by $p_{0,0},\dots,p_{4,0}$. They are all real points.

We then consider the action of the order-$11$ element $P$, which acts as
in Table~\ref{tbl:PR}. This allows us to recover the other $50$ points, which have
coordinates in $\bQ(\zeta)$ and are complex points. We write $p_{i,j}$ for
$P^j(p_{i,0})$. One may then verify that all $55$ points are distinct, and, thus,
the subgroup $\bB$ generated by $P$ and $R$ acts transitively on them.

See Section~\ref{m2:55-points} for the \Macaulay{} code.
\end{proof}

\section{Automorphism group and Picard group}
We consider again the general case. In \cite[Section~4]{bs}, it is shown that
for a trivector $\sigma$ such that $X_6^\sigma$ is smooth, each isolated
singular point $p=[V_1]$ of $X_1^\sigma$, where $\sigma(V_1,-,-)$ has rank $4$
leads to a divisor
\[
D\coloneqq\setmid{[U_6]\in X^\sigma_6}{U_6\supset V_1},
\]
in $X^\sigma_6$. We have the following result \cite[Lemma~4.9 and
Corollary~4.10]{bs}.
\begin{proposition}\label{prop:D24}
Let $\sigma$ be a trivector such that $X_6^\sigma$ is a smooth hyperkähler
fourfold. Let $p=[V_1]$ be an isolated singular point of $X_1^\sigma$ and let $D$
be the induced divisor. Write $H$ for the Plücker polarization on
$X^\sigma_6$. Then the intersection matrix between $H$ and $D$ with respect to
the Beauville--Bogomolov--Fujiki form $\q$ on $H^2(X_6^\sigma,\bZ)$
is
\[\begin{pmatrix} 22 & 2 \\ 2 & -2\end{pmatrix}.\]
The class $D$ has divisibility~$1$.
\end{proposition}
Divisors induced by distinct isolated singular points are also distinct. This
can be proved by computing their intersection numbers as follows.
\begin{proposition}
\label{prop:intersection-number}
Let $\sigma$ be a trivector such that $X_6^\sigma$ is a smooth hyperkähler
fourfold. Let $p=[V_1]$ and $p'=[V_1']$ be two different isolated singular
points on $X_1^\sigma$. Write $D$ and $D'$ for the
divisors on $X_6^\sigma$ that they define. If $\sigma(V_1,V_1',-)=0$, then
$\q(D,D')=1$; otherwise we have $\q(D,D')=0$. In particular, the classes $D$
and $D'$ are distinct.
\end{proposition}
\begin{proof}
The Beauville--Bogomolov--Fujiki form $\q$ has the property
\begin{equation}
\label{eq:BBF}
H^2\cdot D\cdot D'=\q(H,H)\q(D,D')+2\q(H,D)\q(H,D')=22\q(D,D')+8.
\end{equation}
So we shall calculate the degree of the intersection $D\cap D'$ with respect
to the polarization $H$.

If $\sigma(V_1,V_1',-)=0$, the intersection $D\cap D'$ can be identified with the
locus in $\Gr(2,V_6/(V_1+V_1'))\times \Gr(2,V_6'/(V_1+V_1'))$, where $\sigma$
vanishes. A simple computation with \Macaulay{} (see Section~\ref{m2:BBF})
shows that its degree is $30$.
 We may then conclude that
$\q(D,D')=1$ using the relation~\eqref{eq:BBF}.

If $\sigma(V_1,V_1',-)\ne0$, the kernel of this linear form is a subspace $V_9$
such that $V_6+V_6'\subset V_9$. We first show that we have $V_6+V_6'=V_9$. If
the inclusion were strict, we would get a subspace $V_6\cap V_6'$ of dimension
$\ge 4$, which satisfies the vanishing condition $\sigma(V_1+V_1',V_6\cap
V_6',V_{10})=0$. However, the condition $\sigma(V_1,V_1',-)\ne 0$ shows that
$V_1\not\subset V_6'$ and $V_1'\not\subset V_6$, so the intersection of
$V_1+V_1'$ and $V_6\cap V_6'$ is $0$. This means that for every $V_1''$
contained in $V_1+V_1'$, the kernel of $\sigma(V_1'',-,-)$ contains both
$V_1''$ and $V_6\cap V_6'$ and is, therefore, of dimension at least $5$. In
particular, we have $\rk \sigma|_{V_1''}\le 4$ so the entire line
$\bP(V_1+V_1')$ is singular in $X^\sigma_1$, contradicting the hypothesis that
$[V_1]$ and $[V_1']$ are isolated singular points.

So we get $V_6+V_6'=V_9$ and, therefore, $\dim V_6\cap V_6'=3$. A point~$[U_6]$
in the intersection $D\cap D'$ can be given by the following data: first choose
a $2$-plane $U_2$ in $V_6\cap V_6'$, then choose a one-dimensional subspace of
$V_6/(V_1+U_2)$ and another one-dimensional subspace of $V_6'/(V_1'+U_2)$. In
other words, the intersection $D\cap D'$ can be identified as a certain
zero-locus in the fiber product of two projective bundles over $\Gr(2, V_6\cap
V_6')$. A computation with \Macaulay{} (see Section~\ref{m2:BBF})
shows that this is a surface of degree
$8$ with respect to the polarization $H$. We can, thus, conclude that
$\q(D,D')=0$ in the second case, again using the relation~\eqref{eq:BBF}.

Since $\q(D,D)=-2$, we immediately see that different isolated singular points
$p$ induce different divisor classes $D$.
\end{proof}

\begin{remark}
In the case of $\sigma(V_1,V_1',-)=0$, we showed that the intersection $D\cap
D'$ is a surface of degree 30. In fact, if smooth, it is a K3 surface admitting
(at least) two polarizations with intersection matrix
$\begin{psmallmatrix} 6 & 9 \\ 9 & 6 \end{psmallmatrix}$,
and the class $H$ is their sum, which has degree $30$.
\end{remark}

In the case of the special trivector $\bsigma$, we get~$55$ distinct divisors
$D_{i,j}$ on~$X_6^\bsigma$, where $D_{i,j}$ is induced by the isolated singular
point $p_{i,j}$ as given in Proposition~\ref{prop:55-points}.
Since the subgroup $\bB$ acts transitively on the $55$ singular points, we see
that $\bB$ injects into $\Aut^s_H(X_6^\bsigma)$. By the simplicity
of $\GG$, this holds for the whole group $\GG$.
\begin{corollary}
The automorphism group $\Aut^s_H(X_6^\bsigma)$ admits $\GG$ as a subgroup.
\end{corollary}

We now study the Picard group of $X^\bsigma_6$. We will write
$\Pic(X_6^\bsigma)$ for the Picard group and $H^2_{\trans}(X_6^\bsigma)$ for the
transcendental lattice, which is the orthogonal complement of
$\Pic(X_6^\bsigma)$ in $H^2(X_6^\bsigma,\bZ)$.

Since we have the explicit coordinates for the $55$ singular points,
Proposition~\ref{prop:intersection-number} allows us to compute the Gram matrix
between their corresponding divisors. In fact, it suffices to consider the
first 21 singular points $p_{0,0},\dots,p_{0,10},p_{1,0},\dots,p_{1,9}$,
that is, the entire $\inner{P}$-orbit of $p_{0,0}$ plus another $10$ points in
the $\inner{P}$-orbit of $p_{1,0}$. We compute the $21\times 21$
Gram matrix for the corresponding classes
$D_{0,0},\dots,D_{0,10},D_{1,0},\dots,D_{1,9}$ using
Proposition~\ref{prop:intersection-number} (see Section~\ref{m2:pic} for the code)
and obtain the following
\begin{equation}
\label{eq:picX6}
\begin{psmallmatrix}
-2& 1& 0& 1& 0& 0& 0& 0& 1& 0& 1& 0& 0& 1& 0& 0& 0& 0& 0& 0& 1\\
1& -2& 1& 0& 1& 0& 0& 0& 0& 1& 0& 0& 0& 0& 1& 0& 0& 0& 0& 0& 0\\
0& 1& -2& 1& 0& 1& 0& 0& 0& 0& 1& 1& 0& 0& 0& 1& 0& 0& 0& 0& 0\\
1& 0& 1& -2& 1& 0& 1& 0& 0& 0& 0& 0& 1& 0& 0& 0& 1& 0& 0& 0& 0\\
0& 1& 0& 1& -2& 1& 0& 1& 0& 0& 0& 0& 0& 1& 0& 0& 0& 1& 0& 0& 0\\
0& 0& 1& 0& 1& -2& 1& 0& 1& 0& 0& 0& 0& 0& 1& 0& 0& 0& 1& 0& 0\\
0& 0& 0& 1& 0& 1& -2& 1& 0& 1& 0& 0& 0& 0& 0& 1& 0& 0& 0& 1& 0\\
0& 0& 0& 0& 1& 0& 1& -2& 1& 0& 1& 0& 0& 0& 0& 0& 1& 0& 0& 0& 1\\
1& 0& 0& 0& 0& 1& 0& 1& -2& 1& 0& 0& 0& 0& 0& 0& 0& 1& 0& 0& 0\\
0& 1& 0& 0& 0& 0& 1& 0& 1& -2& 1& 1& 0& 0& 0& 0& 0& 0& 1& 0& 0\\
1& 0& 1& 0& 0& 0& 0& 1& 0& 1& -2& 0& 1& 0& 0& 0& 0& 0& 0& 1& 0\\
0& 0& 1& 0& 0& 0& 0& 0& 0& 1& 0& -2& 0& 1& 0& 0& 1& 1& 0& 0& 1\\
0& 0& 0& 1& 0& 0& 0& 0& 0& 0& 1& 0& -2& 0& 1& 0& 0& 1& 1& 0& 0\\
1& 0& 0& 0& 1& 0& 0& 0& 0& 0& 0& 1& 0& -2& 0& 1& 0& 0& 1& 1& 0\\
0& 1& 0& 0& 0& 1& 0& 0& 0& 0& 0& 0& 1& 0& -2& 0& 1& 0& 0& 1& 1\\
0& 0& 1& 0& 0& 0& 1& 0& 0& 0& 0& 0& 0& 1& 0& -2& 0& 1& 0& 0& 1\\
0& 0& 0& 1& 0& 0& 0& 1& 0& 0& 0& 1& 0& 0& 1& 0& -2& 0& 1& 0& 0\\
0& 0& 0& 0& 1& 0& 0& 0& 1& 0& 0& 1& 1& 0& 0& 1& 0& -2& 0& 1& 0\\
0& 0& 0& 0& 0& 1& 0& 0& 0& 1& 0& 0& 1& 1& 0& 0& 1& 0& -2& 0& 1\\
0& 0& 0& 0& 0& 0& 1& 0& 0& 0& 1& 0& 0& 1& 1& 0& 0& 1& 0& -2& 0\\
1& 0& 0& 0& 0& 0& 0& 1& 0& 0& 0& 1& 0& 0& 1& 1& 0& 0& 1& 0& -2
\end{psmallmatrix}.\end{equation}
This matrix is of determinant $22$, a square-free integer, hence the given $21$
classes are linearly independent and generate the whole Picard group. So
$X_6^\bsigma$, indeed, has maximal Picard rank~21.  Using the condition
$\q(H,D)=2$, we may express $H$ in terms of the classes $D_{i,j}$; we obtain
\[
H=D_{0,0}+\cdots+D_{0,10}.
\]
In other words, the polarization $H$ is the sum of the class of $11$ divisors
obtained using the cyclic action of $P$.

Write $H^\perp$ for the orthogonal complement of $H$ in $\Pic(X_6^\bsigma)$,
which is of rank $20$ and negative definite. Its Gram matrix can be explicitly computed. Using the
functionalities for integral lattices in {\tt Sage}, we can verify the following.
\begin{lemma}
\label{lemma:OD}
The lattice $H^\perp$ is of discriminant $121$ and
$\tilde\O(H^\perp)$---the subgroup of isometries of $H^\perp$ acting trivially
on the discriminant group $D(H^\perp)$---is isomorphic to $\GG$.
\end{lemma}
\begin{proof}
By {\tt Sage} computation, see Section~\ref{sage:OD}.
\end{proof}

We now show that the group $\Aut^s_H(X_6^\bsigma)$ of symplectic automorphisms
fixing the polarization $H$ is exactly $\GG$.
\begin{proposition}
\label{prop:aut}
We have $\Aut^s_H(X_6^\bsigma)\simeq\GG$.
\end{proposition}
\begin{proof}
The second cohomology group $\Lambda\coloneqq H^2(X_6^\bsigma,\bZ)$ is
a lattice with discriminant~$2$. The Picard group is a primitive
sublattice of $\Lambda$ of discriminant $22$, which contains the sublattice
$H^\perp$ of discriminant $121$. The orthogonal complement $T$ of $H^\perp$ in
$\Lambda$ must then have discriminant~242. It is the saturation lattice of
the direct sum $H^2_{\trans}(X_6^\bsigma)\oplus\inner{H}$. In particular, we have
$\big|\Lambda/(T\oplus H^\perp)\big|=121=|D(H^\perp)|$.

The transcendental lattice $H^2_{\trans}(X_6^\bsigma)$ is of rank 2 and is contained in
$H^{2,0}(X_6^\bsigma)\oplus H^{0,2}(X_6^\bsigma)$. Therefore, each symplectic
automorphism of $X_6^\bsigma$ fixes $H^2_{\trans}(X_6^\bsigma)$, and an element of
$\Aut^s_H(X_6^\bsigma)$ fixes the sublattice $T$.
Denote by $\O(\Lambda,T)$ the subgroup of isometries of $\Lambda$ fixing the
sublattice~$T$, that is,
\[
\O(\Lambda,T)\coloneqq\setmid{\phi\in\O(\Lambda)}{\phi|_T=\Id_T}.
\]
We get homomorphisms
\[
\GG\into\Aut^s_H(X_6^\bsigma)\into
\O(\Lambda,T)\stackrel{\mathrm{res}}\longrightarrow \O(H^\perp).
\]
In the last homomorphism, since we have the equality
$\big|\Lambda/(T\oplus H^\perp)\big|=|D(H^\perp)|$, we may apply
\cite[Corollary 3.4]{ghs} to show that the image is contained in the subgroup
$\tilde\O(H^\perp)$, which is isomorphic to
$\GG$. So all the inclusions are equalities.
\end{proof}

\begin{remark}\leavevmode
\begin{enumerate}
\item By viewing $\Lambda_\bQ\coloneqq H^2(X_6^\bsigma,\bQ)$ as a
rational $\GG$-representation, we just saw that it decomposes into
\[
\Lambda_\bQ = H^2_{\trans}(X_6^\bsigma)_\bQ\oplus \bQ H\oplus (H^\perp)_\bQ,
\]
where $\GG$ acts trivially on the first two components. For the third
component, again by a computation of characters, it is the direct sum of two
copies of $V_{10}'$ (the other irreducible $\GG$-representation of dimension
10; see Table~\ref{tbl:PSL-char} and Section~\ref{sage:OD}). A geometric
interpretation of this fact would be interesting.

\item Mongardi showed that the fixed locus of a symplectic automorphism $g$ of order $11$
consists of~$5$ isolated points \cite[Proposition~6.2.16]{monthesis}.
As an example, for the automorphism
on~$X_6^\bsigma$ given by the element~$P$, using the eigenvalues
of~$\bw2\rho(P)$ given in Table~\ref{tbl:PR}, we can explicitly determine the
fixed locus as the five points
\[
[024568],[235679],[125689],[015789],[046789],
\]
where the symbol~$[abcde\!f]$ means the six-dimensional
subspace~$V_6=\inner{e_a,\dots,e_f}$ of $V_{10}$.
\end{enumerate}
\end{remark}

We would now like to determine the structures of the various
lattices: the Picard group~$\Pic(X_6^\bsigma)$, the transcendental
lattice~$H^2_{\trans}(X_6^\bsigma)$, and the~$\GG$-invariant sublattice~$T$, which is the
saturation of the direct sum~$H^2_{\trans}(X_6^\bsigma)\oplus\inner{H}$. We recall the
following results from lattice theory~\cite[Corollary~1.13.3 and
Corollary~1.13.5]{nikulin}.
\begin{proposition}[Nikulin]
\label{prop:nikulin}
Let~$L$ be an even lattice of signature~$(p,q)$. Let~$l$ be the minimal number
of generators of the discriminant group~$D(L)$.
\begin{enumerate}
\item If~$p\ge1,q\ge1$, and~$p+q\ge l+2$, then~$L$ is uniquely determined by
its discriminant form.
\item If~$p\ge1,q\ge1$, and~$p+q\ge l+3$, then~$L$ decomposes into~$U\oplus
L'$.
\item If~$p\ge1,q\ge8$, and~$p+q\ge l+9$, then~$L$ decomposes
into~$E_8(-1)\oplus L'$.
\end{enumerate}
Here, $U$ denotes the hyperbolic plane, and $E_8(-1)$ denotes the $E_8$-lattice
with negative definite form.
\end{proposition}

\begin{proposition}
\label{prop:lattices}
Consider the lattice
\[L_{11}\coloneqq\begin{pmatrix}2 & 1 \\1 & 6 \end{pmatrix}.\]
We have the following isomorphisms of lattices
\begin{gather*}
H^2_{\trans}(X_6^\bsigma)\simeq L_{11},\qquad T=H^2(X_6^\bsigma,\bZ)^\GG=H^2_{\trans}(X_6^\bsigma)\oplus\inner{H}\simeq
L_{11}\oplus(22),\\
\Pic(X_6^\bsigma)\simeq
U\oplus E_8(-1)^{\oplus 2}\oplus L(-1),
\end{gather*}
where we can take the component $L$ to be
$\begin{psmallmatrix}2&1&0\\1&2&1\\0&1&8\end{psmallmatrix}$
or $L_{11}\oplus(2)$ (by this we mean the isomorphism holds for both values
of~$L$).
\end{proposition}
\begin{proof}
Since $\Pic(X_6^\bsigma)$ has discriminant~$22$, its orthogonal
$H^2_{\trans}(X_6^\bsigma)$ has discriminant
either~$11$ or~$44$. In the second case, the direct sum $H^2_{\trans}(X_6^\bsigma)\oplus
\inner{H}$ would have index~$2$ in its saturation~$T$, so there would exist a
class $x\in H^2_{\trans}(X_6^\bsigma)$ such that~$\frac12(H+x)$ is integral.
However, then we would have $\q\big(H,\frac12(H+x)\big)=11$, contradicting the fact
that $\div(H)=2$.

So~$H^2_{\trans}(X_6^\bsigma)$ has discriminant~$11$. Every rank-$2$ positive definite
lattice has a reduced form (see, for instance,~\cite[Chapter~15.3.2]{conwaysloane}). For discriminant~$11$, the
lattice~$L_{11}$ is the only one that is even. Thus, we may conclude that
$H^2_{\trans}(X_6^\sigma)$ is isomorphic to~$L_{11}$. Since the direct
sum~$H^2_{\trans}(X_6^\bsigma)\oplus\inner{H}$ is primitive, we
have~$T=H^2_{\trans}(X_6^\bsigma)\oplus \inner{H}\simeq L_{11}\oplus(22)$.

Finally we determine the structure of $\Pic(X_6^\bsigma)$. By using (2) and (3)
of Proposition~\ref{prop:nikulin}, it decomposes into a direct sum
\begin{equation}
\label{eq:picX6-iso-type}
\Pic(X_6^\bsigma)\simeq U\oplus E_8(-1)^{\oplus 2}\oplus L(-1),
\end{equation}
where~$L$ is positive definite of rank~$3$ and discriminant~$22$ and also even.
There are two possibilities: either~$L$ is indecomposable, then by
\cite[Table 15.6]{conwaysloane} it is unique and is given by
\[
\begin{pmatrix}2&1&0\\1&2&1\\0&1&8\end{pmatrix};
\]
or~$L$ is decomposable, then it must be the direct sum~$L_{11}\oplus(2)$. By
comparing discriminant forms and using (1) of Proposition~\ref{prop:nikulin}, we may
conclude that the isomorphism \eqref{eq:picX6-iso-type} holds for both values
of $L$.
\end{proof}
We next prove the uniqueness of our special Debarre--Voisin fourfold,
following the same idea of \cite[Theorem~4.2]{dm}.
The following general result is important (see \cite[Example~2.5.9, Section~7.4.4]{monthesis} and~\cite[Theorem~A.3]{dm}).
\begin{theorem}[Mongardi]
\label{thm:possible-T}
Let $X$ be a hyperkähler fourfold of $\KKK^{[2]}$-type that admits a symplectic
automorphism~$g$ of order~$11$. The Picard rank of such a fourfold
is equal to the maximal value $21$.
There are two possibilities for the~$g$-invariant sublattice $T$ of $H^2(X,\bZ)$
\begin{equation}\label{eq:possible-T}
\begin{pmatrix}
2 & 1 & 0\\
1 & 6 & 0\\
0 & 0 & 22
\end{pmatrix}
\quad\text{ or }\quad
\begin{pmatrix}
6 & 2 & 2\\
2 & 8 & -3\\
2 & -3 & 8
\end{pmatrix},
\end{equation}
and the $g$-coinvariant sublattice $T^\perp$ is always isomorphic to an explicitly described lattice $\bS$.
\end{theorem}
Since $\GG$ contains elements of order 11, this states that the lattice $H^\perp$ in our
case must be isomorphic to $\bS$ (see Section~\ref{sage:OD} for the Gram matrix).

We also need the following lemma, which is essentially \cite[Lemma~4.3]{dm}.
\begin{lemma}\label{lemma:uniqueness}
Denote by $\Lambda$ the lattice $H^2(X,\bZ)$ where $X$ is of $\KKK^{[2]}$-type. Let $T$ be either one of the two lattices in \eqref{eq:possible-T}. Up to the action of $\O(\Lambda)$, there is a unique primitive embedding of $T$ into $\Lambda$ such that the orthogonal complement is isomorphic to $\bS$.
\end{lemma}
\begin{proof}
This is a direct application of \cite[Proposition~1.15.1]{nikulin}: since the orthogonal complement $T^\perp$ has been prescribed, up to the action of $\O(\Lambda)$, a primitive embedding of $T$ in $\Lambda$ is equivalent to the data of subgroups $K_T\subset D(T)\simeq \bZ/11\bZ\oplus \bZ/22\bZ$ and $K_\Lambda\subset D(\Lambda)\simeq \bZ/2\bZ$, together with an isometry $K_T\simto K_\Lambda$ for the induced $\bQ/2\bZ$-valued quadratic forms. Moreover, the orthogonal complement has discriminant $484/\operatorname{Card}(K_\Lambda)^2$. Since the lattice $\bS$ has discriminant $121$, we see that $K_T$ must be the $2$-torsion subgroup while $K_\Lambda=D(\Lambda)$, and the isometry $K_T\simto K_\Lambda$ is also unique.
\end{proof}

\begin{proposition}\label{prop:uniqueness}
The fourfold $X_6^\bsigma$ can be characterized as the unique polarized hyperkähler fourfold $(X,H)$ of $\KKK^{[2]}$-type with $H$ of square $22$ and divisibility $2$ that admits a symplectic automorphism of order~$11$ fixing the polarization $H$.
\end{proposition}
\begin{proof}
Let $(X,H)$ be a such pair. Clearly, $H$ lies in the $g$-invariant sublattice $T$. Since the two lattices in \eqref{eq:possible-T} are both positive definite, one can easily verify that the first lattice contains $(0,0,\pm1)$ as the only vectors with square $22$ and divisibility $2$, while the second does not contain any such vector. So one may conclude that $H^2_\trans(X)\simeq L_{11}$, and $T$ is given by $T\simeq L_{11}\oplus(22)$.

Now for two such pairs $(X,H)$ and $(X',H')$, by Lemma~\ref{lemma:uniqueness} there exists an isometry $\phi\colon H^2(X,\bZ)\to H^2(X',\bZ)$ that maps the invariant sublattice $T$ to $T'$. Up to passing to $-\phi$, we may, moreover, assume that $\phi$ maps $H$ to $H'$ and $H^2_\trans(X)$ to $H^2_\trans(X')$. Then $\phi$ is either Hodge or \emph{anti}-Hodge, that is, $\phi(H^{2,0}(X))=H^{0,2}(X')$. In the latter case, we may take the vector $u\in H^2_\trans(X)$ with square $2$ and consider the reflection $R_u$, which reverses the orientation of $H^2_\trans(X)$. The composition $\phi\circ R_u$ would then be a Hodge isometry. So we always get a Hodge isometry mapping $H$ to $H'$.
In other words, for a fixed element $h\in\Lambda$ with square $22$ and divisibility $2$, by picking markings that map both $H$ and $H'$ to $h$, the period points of $X$ and $X'$ will lie in the same $\O(\Lambda,h)$-orbit.
By the global Torelli theorem for polarized hyperkähler manifolds~\cite[Theorem~8.4]{markman} (see also Section~\ref{sec:HLS}), we conclude that $(X,H)$ is isomorphic to $(X',H')$.
\end{proof}
\begin{remark}\leavevmode
\label{rmk:end}
\begin{enumerate}

\item O'Grady~\cite[Theorem~1.8]{ogrady} showed that the modular map $\mathfrak m\colon\MDV\dashrightarrow \cM_{22}^{(2)}$ from the GIT moduli space of trivectors to the moduli space of polarized hyperkähler manifolds is of degree $1$. One can improve on this result to obtain a precise statement on the injectivity (see Proposition~\ref{prop:m_immersion}). Therefore, the uniqueness of $X_6^\bsigma$ up to isomorphism can also be used to deduce the uniqueness of the trivector $\bsigma$ up to linear transformations.

\item In the Introduction we mentioned the existence of two other hyperkähler
fourfolds of $\KKK^{[2]}$-type with a $\GG$-action, constructed using the
variety of lines for the Klein cubic fourfold and a special double EPW sextic
\cite[Section~4.5]{monthesis}; we denote them by $\Fano$ and $\EPW$, respectively, and refer to our example $X_6^\bsigma$ also as $\DV$. The three examples are all distinct, which can
be shown by comparing the transcendental lattices: the transcendental lattice of $\Fano$ is isomorphic to $\begin{psmallmatrix}22
&11\\11&22\end{psmallmatrix}$ \cite[Theorem~1.8]{LazaZheng} (the result is
stated for the cubic so there is a change of sign),
while the transcendental lattice of $\EPW$ is isomorphic
to $(22)^{\oplus 2}$ \cite[Corollary~A.4]{dm}.

\item The uniqueness of the other two examples can be similarly proved using the same argument from Proposition~\ref{prop:uniqueness}. Note that the two possibilities for the invariant sublattice $T$ can also be interpreted as
two twistor families parametrizing pairs $(X,\omega)$, where $X$ is of $\KKK^{[2]}$-type, and $\omega$ is a Kähler class on $X$ invariant by some symplectic automorphism $g$ of order~$11$ (see~\cite[Section~7.4.4]{monthesis}). Hence, $\DV$ and $\EPW$ lie in the same twistor family, while $\Fano$ belongs to another twistor family.

\item In all three cases, the coinvariant sublattice $H^\perp$
is of the same isomorphism type $\bS$ and has discriminant 121. However, the
square of the polarization $H$ takes different values: 22 for $\DV$, 6 for $\Fano$, 2 for $\EPW$.
So for the latter two, the Picard lattice splits as a direct sum $\inner{H}\oplus
\bS$, while for $\DV$ the direct sum $\inner{H}\oplus \bS$ is of index $11$ in $\Pic(X_6^\bsigma)$.
\end{enumerate}
\end{remark}

Recall that a K3 surface with maximal Picard rank $20$ is isolated in the moduli and is uniquely determined by its transcendental lattice, by a result of Shioda and Inose \cite{shioda}.
A similar argument as above shows the following result.
\begin{proposition}\label{prop:S2}
The special Debarre--Voisin fourfold $X_6^\bsigma$ is birationally isomorphic to $S^{[2]}$, where $S$ is the unique K3 surface of maximal Picard rank $20$ with transcendental lattice $L_{11}=\begin{psmallmatrix}2&1\\1&6\end{psmallmatrix}$.
\end{proposition}
\begin{proof}
First, we note that the transcendental lattice of $S^{[2]}$ is also $L_{11}$.
Another application of \cite[Proposition~1.15.1]{nikulin} shows that the primitive embedding of $L_{11}$ into $\Lambda$ is unique up to the action of $\O(\Lambda)$ (see also \cite[Proposition~2.7]{BCS}).
So again, up to composing with the reflection $R_u$ where $u$ has square $2$, there exists a Hodge isometry between $H^2(X_6^\bsigma,\bZ)$ and $H^2(S^{[2]},\bZ)$. The global Torelli theorem then affirms that $X_6^\bsigma$ and $S^{[2]}$ are birationally isomorphic to each other.
\end{proof}

We finish by determining the values $d$ for which $X_6^\bsigma$ is special of discriminant $d$ in the sense of Hassett. Recall from \cite[Section~4]{debmacri} that a polarized hyperkähler fourfold $(X,H)$ of $\KKK^{[2]}$-type is called \emph{Hassett special of discriminant $d$}, if there exists a primitive rank-$2$ sublattice $K\subset \Pic(X)$ containing $H$ such that $\disc(K^\perp)=-d$ (the orthogonal is taken in $H^2(X,\bZ)$). Moreover, one has $\disc(K)=-2d$ when $H$ has divisibility $2$, and $\disc(K)=-d/2$ when $H$ has divisibility $1$. So for example, by Proposition~\ref{prop:D24} we know that $X_6^\bsigma$ is Hassett special of discriminant $24$.

For the special double EPW sextic $\EPW$ and the variety of lines $\Fano$, since their Picard lattices split as a direct sum $\inner{H}\oplus \bS$, it is quite straightforward to determine the possible discriminants (see~\cite[Proposition~A.5]{dm}): the lattice $\bS$ primitively represents all even integers $2k$ for $k\le -2$, so $\EPW$ is Hassett special for all $d\ge 16$ that is a multiple of $8$. Using the same argument, one can check that $\Fano$ is Hassett special for all $d\ge 12$ that is a multiple of $6$.

In our case, since the Picard lattice does not split as a direct sum, the situation is slightly more complicated. We first note that for a hyperkähler fourfold of $\KKK^{[2]}$-type with a polarization of square $22$ and divisibility $2$, the possible discriminants $d$ are the positive even integers, such that $d/2$ is a square modulo 11 and $d\ne 22$ \cite[Proposition~4.1 and Theorem~6.1]{debmacri}.
Moreover, it was shown in \cite{DHOV} and \cite{oberdieck} that the discriminants $2,6,8,10,18$ cannot arise for the Debarre--Voisin construction (see also Section~\ref{sec:HLS}). So a smooth Debarre--Voisin fourfold can only be Hassett special of discriminant $d$ for the following values
\begin{equation}\label{eq:mod22}
\tag{$*$}
d\ge 24 \quad\text{and}\quad d\equiv 0,2,6,8,10,18\pmod{22}.
\end{equation}
In contrast to the cases of $\EPW$ and $\Fano$, the special Debarre--Voisin $X_6^\bsigma$ turns out to be Hassett special for all possible discriminants.
\begin{proposition}\label{prop:disc}
The Debarre--Voisin fourfold $X_6^\bsigma$ is Hassett special for all possible discriminants $d\ge 24$, that is, all integers $d$ satisfying \eqref{eq:mod22}.
\end{proposition}
\begin{proof}
Recall the Gram matrix~\eqref{eq:picX6} for the Picard lattice $\Pic(X_6^\bsigma)$ equipped with the Beauville--Bogomolov--Fujiki form $\q$. We consider a second quadratic form $Q$ given by
\[
Q\colon u\mapsto -\frac12\det\begin{pmatrix}\q(H,H)&\q(H,u)\\\q(H,u)&\q(u,u)\end{pmatrix}=\frac12\q(H,u)^2-11\q(u,u).
\]
If $u$ is primitive, and the sublattice $K=\inner{u}\oplus\inner{H}$ is saturated in $\Pic(X_6^\bsigma)$, then $\disc(K)=-2Q(u)$, so $X_6^\bsigma$ is Hassett special of discriminant $Q(u)$. Hence, we need to show that for all possible $d$, the quadratic form $Q$ primitively represents $d$ with a such vector $u$. We will prove this for $d$ large enough and manually check all the small cases. Note that $Q$ is only positive semi-definite since it vanishes for $H$.

Consider the sublattice $L$ of $\Pic(X_6^\bsigma)$ generated by the following seven elements
\[
\begin{gathered}
D_{0,0},\quad D_{0,1}-D_{1,9},\quad D_{0,5}-D_{1,8},\quad D_{0,6}-D_{1,1},\quad D_{0,7}-D_{1,4},\\
D_{0,4}+D_{0,9}-D_{1,1}+D_{1,2}-D_{1,3}-D_{1,9},\quad D_{1,1}-D_{1,4}+D_{1,5}-D_{1,6}+D_{1,7}-D_{1,8}.
\end{gathered}
\]
Note that none of the elements has a component in $D_{0,2}$.
Recall that $H$ is equal to $D_{0,0}+\cdots +D_{0,10}$.
Hence, by examining the coefficient before $D_{0,2}$, one can see that for any primitive $u\in L$, the sublattice $K=\inner{u}\oplus\inner{H}$ is, indeed, saturated.
Therefore, it suffices to show that $Q|_L$ primitively represents all large enough $d$. One computes that $Q|_L$ is the diagonal form $(24,44,44,44,44,66,66)$ in the given basis. By Lagrange's four-square theorem, all positive integers $44k$ are represented by the form $(44,44,44,44)$. By adding the last two coordinates $(66,66)$, all integers $22k$ with $k\ge 2$ can be primitively represented. Together with the first coordinate, we obtain all integers of the form $24a^2+22k$, for $a\in\set{0,\dots,5}$ and $k\ge 2$. These cover all six residue classes modulo 22, hence $Q|_L$ primitively represents all $d\ge 24\cdot 5^2+44=644$ satisfying \eqref{eq:mod22}. We conclude by manually checking the smaller cases.
\end{proof}

\section{HLS divisors}
\label{sec:HLS}

Proposition~\ref{prop:disc} has a nice implication on the global geometry of the moduli space of Debarre--Voisin fourfolds, in terms of Hassett--Looijenga--Shah (HLS) divisors. We briefly explain some necessary background.

Consider the $20$-dimensional GIT moduli space of trivectors
\[
\MDV\coloneqq\bP(\bw3V_{10}^\vee)\git\SL(V_{10}),
\]
as well as the moduli space $\cM_{22}^{(2)}$ of polarized hyperkähler manifolds with square $22$ and divisibility $2$. The Debarre--Voisin construction gives a rational modular map
\[
\mathfrak m\colon\MDV\dashrightarrow \cM_{22}^{(2)}, \quad [\sigma]\mapsto [X_6^\sigma].
\]
Debarre and Voisin showed that $\m$ is dominant and generically finite, and O'Grady later showed that it is birational. For our purpose, we need the stronger statement that $\m$ is actually an open immersion when restricted to $\MDVsm$, the open locus consisting of classes of trivectors $[\sigma]$ such that the corresponding $X_6^\sigma$ is smooth of dimension $4$ (recall from Lemma~\ref{lemma:smoothness-X3-X6} that this is precisely equivalent to the fact that $[\sigma]\in \bP(\bw3V_{10}^\vee)$ does not lie on the projective dual $\Delta$ of $\Gr(3,V_{10})$, also known as the discriminant hypersurface).
In fact, this will follow without much difficulty from the same argument in \cite{dv}. We first prove the following intermediate results.

\begin{proposition}
Denote by $G$ the Grassmannian $\Gr(6,V_{10})$ and let $\cU$ and $\cQ$
be the tautological subbundle and quotient bundle on $G$, respectively.
Let $\sigma$ be a trivector such that $X=X_6^\sigma$ is a smooth Debarre--Voisin fourfold. In other words, we assume that $[\sigma]$ does not lie on the discriminant hypersurface $\Delta\subset \bP(\bw3V_{10}^\vee)$.
\begin{enumerate}
\item\label{item:UQ-simple} The restrictions $\cU|_X$ and $\cQ|_X$ are both simple, that is, $\End(\cU|_X)\simeq\End(\cQ|_X)\simeq\bC$. We also have $h^0(X, \cU^\vee|_X)=10$.
\item\label{item:kernel-stab} Denote by $\Stab(\sigma)$ the stabilizer subgroup of $\sigma$ in $\SL(V_{10})$. Consider the natural homomorphism
\[
\Phi\colon\Stab(\sigma)\to \Aut(X_6^\sigma),
\]
which maps each $\varphi\in \SL(V_{10})$ to the induced automorphism $\Phi(\varphi)$ on $X_6^\sigma$. Then the kernel of the map is equal to $\setmid{\lambda \Id}{\lambda^{10}=1}$.
\item\label{item:sigma_stable} The trivector $\sigma$ is stable with respect to the $\SL(V_{10})$-action. In other words, it has a finite stabilizer subgroup $\Stab(\sigma)$, and its (affine) orbit $O(\sigma)$ in $\bw3 V_{10}^\vee$ is closed.
\end{enumerate}
\end{proposition}

\begin{proof}
Statement \eqref{item:UQ-simple} follows from a standard computation using the Koszul complex and the Borel--Weil--Bott theorem. Namely, let $\cF$ be the vector bundle $\bw3\cU^\vee$ of rank $20$ on $G$ and consider the Koszul complex
\[
0\to \bw{20}\cF^\vee\to\cdots\to\bw2\cF^\vee\to\cF^\vee\to\cO_G\to\cO_X\to 0,
\]
which is a free resolution of the structure sheaf of $\cO_X$. Given a vector bundle $\cE$ on $G$, we tensor the complex with $\cE$ and obtain a spectral sequence
\[
E_1^{-q,p}\coloneqq H^p(G,\cE\otimes\bw q \cF^\vee)\Rightarrow H^{p-q}(X,\cE|_X).
\]
We first take $\cE$ to be $\cU^\vee\otimes \cU$ to illustrate the idea.
Each term $H^p(G,\cU^\vee\otimes\cU\otimes\bw q\cF^\vee)$ in the spectral sequence can be computed using the Borel--Weil--Bott theorem, and
one may verify that there are only three terms that are non-zero
\[
h^0(G,\cE)=h^{24}(G,\cE\otimes\det\cF)=1,\quad
h^{12}(G,\cE\otimes \bw{10}\cF)=101.
\]
In particular, the spectral sequence degenerates at the first page, so we conclude that
\[
h^0\big(X,\cE|_X\big)=h^4\big(X,\cE|_X\big)=1,\quad
h^2\big(X,\cE|_X\big)=101,
\]
while $h^1=h^3=0$. Hence, $\dim \End(\cU|_X)=1$ and $\cU|_X$ is, indeed, simple.

Similarly, we can take $\cE$ to be $\cQ^\vee\otimes \cQ$ and $\cU^\vee$ and carry out the same computation for the other two claims.

For statement \eqref{item:kernel-stab}, suppose that $\varphi\in \SL(V_{10})$ induces the trivial automorphism on $X=X_6^\sigma$. Then it will also induce an automorphism
\[
f_\varphi\in \End(\cU|_X),
\]
where the action is fiberwise. However, since the vector bundle $\cU|_X$ is simple, up to multiplying by a non-zero scalar, $f_\varphi$ must be the identity map. In other words, $\varphi$ acts as the identity on each $\bP(V_6)$ for $[V_6]\in X$.

To conclude, note that since $h^0(X,\cU^\vee|_X)=10$, all the six-dimensional vector spaces $[V_6]\in X$ span the entire $V_{10}$, so $\varphi$ acts as the identity on $\bP(V_{10})$ and is, therefore, of the form $\lambda\Id$ for some $\lambda$ with $\det(\lambda\Id)=\lambda^{10}=1$.

Finally, we prove \eqref{item:sigma_stable}. For any $[\sigma]\in \bP(\bw3V_{10}^\vee)$ not lying on the discriminant $\Delta$, since $X_6^\sigma$ is a hyperkähler fourfold, its automorphism group is discrete, hence so is $\Stab(\sigma)$. However, $\Stab(\sigma)$ is an algebraic subgroup, so it is necessarily finite. Consequently, the affine $\SL(V_{10})$-orbit $O(\sigma)\subset\bw3V_{10}^\vee$ has the expected codimension $120-99=21$.

We claim that the orbit $O(\sigma)$ is closed. Suppose that this is not the case, then the closure $\overline{O(\sigma)}$ contains another orbit $O(\sigma')$, which must have higher codimension. In particular, $[\sigma']$ must necessarily lie in $\Delta$. On the other hand, $\Delta$ is defined by the discriminant, an $\SL(V_{10})$-invariant polynomial that is constant on each affine orbit. This implies that it vanishes at $\sigma'$, and, hence, at $\sigma$ as well, which is a contradiction.
\end{proof}
Note that the last point \eqref{item:sigma_stable} ensures that the map
\[
\pi \colon \bP(\bw3V_{10}^\vee)\setminus\Delta \to \MDV
\]
is a geometric quotient, and we denote its image by $\MDVsm$. The modular map $\m$ is regular on the open locus $\MDVsm$.

The following is proved in \cite[Lemma~4.6]{dv}.
\begin{lemma}[Debarre--Voisin]
\label{lemma:diff_m}
Write $(X,H)$ for a Debarre--Voisin fourfold $X_6^\sigma$ and its canonical polarization. Whenever
$X$ is of dimension $4$, any first-order deformation of the pair $(X, H)$ is
given by a deformation of $\sigma$. More precisely, the Kodaira--Spencer map
\[
\KS\colon \bw3V_{10}^\vee /\inner\sigma\to \Def_{(X,H)}(\bC[\varepsilon])
\simeq \Ext^1\big(\cP_{X,H},\cO_X\big)
\]
is surjective, where the bundle $\cP_{X,H}$ is the extension
\[
0\to \Omega_X\to \cP_{X,H}\to \cO_X\to 0
\]
given by $c_1(H)\in \Ext^1(\cO_X,\Omega_X)$.
\end{lemma}

\begin{proposition}\label{prop:m_immersion}
The modular map $\m$ when restricted to $\MDVsm$ induces an open immersion
\[
\m\colon \MDVsm\into \cM_{22}^{(2)}.
\]
\end{proposition}
\begin{proof}
We consider the composition map
\[
\bP(\bw3 V_{10}^\vee)\setminus \Delta \stackrel\pi{\to} \MDVsm \stackrel\m{\to} \cM_{22}^{(2)}.
\]
Lemma~\ref{lemma:diff_m} states precisely that the differential of the map $\m\circ \pi$ is everywhere surjective.

We claim that the restriction $\m|_{\MDVsm}$ is quasi-finite. Otherwise, suppose that a curve $C$ gets contracted; then for each point $[\sigma]$ in $C$, its preimage by $\pi$ in $\bP(\bw3V_{10}^\vee)$ is the $\SL(V_{10})$-orbit of $\sigma$, which has expected codimension $20$. Thus, the entire preimage of $C$ by $\pi$ has codimension $19$ and is contracted by $\m\circ\pi$. This contradicts the surjectivity of the differential.

O'Grady showed that $\m$ is birational \cite[Theorem~1.8]{ogrady}. Moreover, the target space $\cM_{22}^{(2)}$ is normal (this follows from, for example, the normality of the period domain $\cP_{22}^{(2)}$ and the global Torelli theorem, see below). Hence, by Zariski's Main Theorem, $\m|_{\MDVsm}$ is an open immersion.
\end{proof}

We now recall a few notions on the period domain and the global Torelli theorem.
Let $(\Lambda,q)$ be the lattice $H^2(X,\bZ)$ and fix an element $h\in \Lambda$ with square $22$ and divisibility $2$. The period domain $\cP_{22}^{(2)}$ is defined as the quotient of the domain
\[
\setmid{[x]\in \bP(\Lambda\otimes \bC)}{q(x,x)=q(x,h)=0,q(x,\overline{x}) >0},
\]
which has two connected components, by the action of $\O(\Lambda, h)$. It is a normal quasi-projective variety by the Baily--Borel theory. For an element $\kappa \in h^\perp$ with negative square, one defines the associated Heegner divisor to be the hypersurface in the period domain induced by the hyperplane $\kappa^\perp$. It is said to be of discriminant $d$ if $\disc(\kappa^\perp)=-d$ (the orthogonal complement is taken in $h^\perp$). The Heegner divisor $\cD_d$ is non-empty if and only if $d$ is a positive even integer, such that $d/2$ is a square modulo $11$, and, in this case, it is always irreducible~\cite[Proposition~4.1]{debmacri}.
Note that the period point of a Debarre--Voisin fourfold $X_6^\sigma$ lies on the Heegner divisor $\cD_d$ precisely when
it is Hassett special of discriminant $d$.

The global Torelli theorem for polarized hyperkähler manifolds \cite[Theorem~8.4]{markman} states that the period map
\[
\p\colon \cM_{22}^{(2)}\to\cP_{22}^{(2)}
\]
is an open immersion, and the image is the complement of the Heegner divisor $\cD_{22}$ \cite[Theorem~6.1]{debmacri}. One may consider the composition $\p\circ \m$
and resolve its indeterminacies by passing to the Baily--Borel compactification $\overline{\cP_{22}^{(2)}}$ of the period domain (whose boundary has dimension $1$)
\[
\begin{tikzcd}
\MDV\ar[r,dashed,"\p\circ\m"] & \cP_{22}^{(2)}\ar[d,hook]\\
\tilde{\MDV}\ar[u,"\varepsilon"]\ar[r, "\tilde\p"] & \overline{\cP_{22}^{(2)}}.
\end{tikzcd}
\]
In this way we obtain a birational morphism $\tilde\p$.
\begin{definition}\label{def:HLS}
An irreducible divisor in $\cP_{22}^{(2)}$ (or $\cM_{22}^{(2)}$) is called \emph{Hassett--Looijenga--Shah} (HLS) if its closure in $\overline{\cP_{22}^{(2)}}$ is the image of an exceptional divisor of $\varepsilon$ by the extended period map $\tilde\p$.
\end{definition}

The Heegner divisors $\cD_d$ for $d\in\set{2,6,8,10,18}$ are HLS: their geometric description was first studied for all cases except $d=8$ by Debarre, Han, O'Grady, and Voisin \cite{DHOV}, and Oberdieck later provided another proof using Gromov--Witten techniques \cite{oberdieck}. Moreover, it was remarked in \cite{oberdieck} that it is possible to check that these are the only five HLS Heegner divisors. Proposition~\ref{prop:disc} now gives us an alternative proof for this last statement.

\begin{corollary}
\label{cor:HLS}
An HLS divisor does not contain the period point of any smooth Debarre--Voisin fourfold.
Consequently, a Heegner divisor $\cD_d$ is HLS if and only if $d\in\set{2,6,8,10,18}$.
\end{corollary}
\begin{proof}
Suppose that $\cD$ is an HLS divisor containing the period point of a smooth Debarre--Voisin fourfold. Then the intersection $\cD\cap \MDVsm$ is non-empty, and the preimage of $\cD$ by $\tilde\p$ contains its strict transform $\cD'$ in $\tilde{\MDV}$.

However, by definition there is an exceptional divisor $\cD''$ of $\varepsilon$ that dominates $\cD$ via $\tilde\p$. The image $\varepsilon(\cD'')$ does not intersect $\MDVsm$ where $\m$ is regular so $\cD''\ne \cD'$. Hence, the preimage of $\cD$ contains at least two irreducible divisors, both dominating $\cD$. Then the generic point of $\cD$ would have at least two preimages, contradicting the normality of the period domain.

To conclude, Proposition~\ref{prop:disc} implies that for any possible $d\ge 24$, the Heegner divisor $\cD_d$ contains the period point of $X_6^\bsigma$. It is, therefore, not HLS.
\end{proof}

\appendix

\section{Macaulay2 and Sage code}

We provide the computer algebra code used in the note. Only the last
computation for the group of isometries is done in {\tt Sage}, all the others
are in \Macaulay{}.

\subsection{Proposition~\ref{prop:smoothness}: the smoothness of
\texorpdfstring{$X_3^\bsigma$}{X3}}
\label{m2:smoothness}
The following verifies the smoothness of~$X_3^\bsigma$ by using the Jacobian
criterion on each affine chart $\bA^{21}$ of $\Gr(3,V_{10})$.
\begin{lstlisting}[language=Macaulay2]
-- the trivector has ten components
comps = {(0,2,5),(1,3,6),(2,4,7),(3,0,8),(4,1,9),
         (0,9,7),(1,5,8),(2,6,9),(3,7,5),(4,8,6)};
sigma = (u,v,w) -> sum(for idx in comps list det submatrix(u||v||w, idx));
-- check that the hyperplane section is smooth on each chart of Gr(3,10)
S = QQ[g_0..g_20]; -- each chart has 21 coordinates
I3 = id_(S^3); M = genericMatrix(S,3,7);
-- generate the coordinates matrix for the chart ijk
chart = (i,j,k) ->
    M_{0..i-1}|I3_{0}|M_{i..j-2}|I3_{1}|M_{j-1..k-3}|I3_{2}|M_{k-2..6};
isSmooth = true;
for ijk in subsets(10,3) do (Mijk = chart toSequence ijk;
    if dim singularLocus ideal sigma(Mijk^{0},Mijk^{1},Mijk^{2}) >= 0
        then isSmooth = false;);
assert(isSmooth);
\end{lstlisting}
\subsection{Ideals generated by Pfaffians}
\label{m2:ideals}
The following computes the ideals generated by Pfaffians of $\bsigma$ seen as a
$10\times 10$ skew-symmetric matrix.
\begin{lstlisting}[language=Macaulay2]
-- we study the singular locus of the Peskine X1 in P^9
S = QQ[x_0..x_9];
-- compute the skew-symmetric matrix of sigma
delta = (x,y,ex) -> (table(10,10,(i,j) -> if i==x and j==y then ex else 0));
skew = (i,j,k) -> (delta(i,j,x_k)+delta(j,k,x_i)+delta(k,i,x_j)
                  -delta(j,i,x_k)-delta(k,j,x_i)-delta(i,k,x_j));
sigmaskew = matrix sum(for idx in comps list skew(idx));
I2 = trim pfaffians_4 sigmaskew; -- irrelevant ideal
I4 = trim pfaffians_6 sigmaskew; -- 55 points
I6 = trim pfaffians_8 sigmaskew; -- the Peskine X1
<< "{rank<=2}: dim = " << dim I2-1 << ", deg = " << degree I2 << endl;
<< "{rank<=4}: dim = " << dim I4-1 << ", deg = " << degree I4 << endl;
<< "{rank<=6}: dim = " << dim I6-1 << ", deg = " << degree I6 << endl;
\end{lstlisting}
The output is the following.
\begin{lstlisting}
{rank<=2}: dim = -1, deg = 11
{rank<=4}: dim = 0, deg = 55
{rank<=6}: dim = 6, deg = 15
\end{lstlisting}

\subsection{Proposition~\ref{prop:55-points}: the 55 isolated singular points}
\label{m2:55-points}
We first compute the coordinates of the 55 isolated singular points directly
using the {\tt RationalPoints2} package.
\begin{lstlisting}[language=Macaulay2]
needsPackage "RationalPoints2";
assert(#unique rationalPoints(I4, Projective=>true, Split=>true) == 55);
\end{lstlisting}
Then we perform the step-by-step procedure as explained in the proof of
Proposition~\ref{prop:55-points}.
\begin{lstlisting}[language=Macaulay2]
-- use the hyperplane x_0+x_1+x_2+x_3+x_4 to identify 5 points
fivePts = I4 : (I4 : (x_0+x_1+x_2+x_3+x_4));
-- get a degree 5 polynomial in x_0 and x_1 using elimination
pol5 = (gens gb sub(fivePts, QQ[x_0..x_9, Weights=>(2:0)|(8:1)]))_(0,0);
-- the polynomial will split in Q(zeta) so we take a field extension
F = toField(QQ[z]/((z^11-1)//(z-1))); S' = F[x_0..x_9];
root = ideal(x_1 - (z^7+z^6+z^5+z^4) * x_0); -- take a root
assert zero(sub(pol5, S') % root); -- verify that it is a root
fivePts' = sub(fivePts, S');
Ip = saturate(fivePts', saturate(fivePts', root)); -- ideal of one of the points
-- we recover the coordinates of p by solving a linear system in x_0,...,x_9
coeffs = f->first entries transpose(coefficients(f, Monomials=>gens S'))_1;
mat = matrix({{1,9:0}}|apply(first entries gens Ip, coeffs));
p = first entries transpose sub((inverse mat)_{0}, F);
-- finally compute the orbit of p
coord = p -> apply(p, x->x//p_0); -- compute the coordinate (1:x_1:x_2:...:x_9)
Peigen = (for a in (10,2,7,8,6,5,1,9,4,3) list z^a);
P = (j, p) -> coord apply(10, i->p_i*Peigen_i^j); -- P acts by scaling
R = p -> coord{p_1,p_2,p_3,p_4,p_0,p_6,p_7,p_8,p_9,p_5}; -- R acts by permuting
pts = flatten apply({p, R p, R R p, R R R p, R R R R p}, p->apply(11, j -> P_j p));
assert(#unique pts == 55); -- G acts transitively on all 55 singular points
\end{lstlisting}
\subsection{Proposition~\ref{prop:intersection-number}: the Beauville--Bogomolov--Fujiki form}
\label{m2:BBF}
The following verifies the intersection numbers $D\cdot D'\cdot H^2$ for the
two divisors $D$ and $D'$ in the two cases that are studied in the proof of
Proposition~\ref{prop:intersection-number}.
\begin{lstlisting}[language=Macaulay2]
needsPackage "Schubert2";
-- first case: D intersect D' is a K3 of degree 30
G = flagBundle{2,2}; U1 = first bundles G;
GxG = flagBundle({2,2}, 4*OO_G); U2 = first bundles GxG;
X = sectionZeroLocus dual(det U1*(1+U2) + det U2*(1+U1));
(h1, h2) = chern_1 \ (dual U1*OO_X, dual U2*OO_X);
print (integral \ (h1^2, h1*h2, h2^2)); -- (6, 9, 6)
-- second case: D intersect D' is a surface of degree 8
G = flagBundle{2,1}; (U,Q) = bundles G;
P1 = flagBundle({1,2}, 2+Q); (U1,Q1) = bundles P1;
P2 = flagBundle({1,2}, 2+Q*OO_P1); (U2,Q2) = bundles P2;
X = sectionZeroLocus dual(det U*(U1+U2)+U*U1*U2);
h = chern_1 (dual(U+U1+U2)*OO_X);
print integral h^2; -- 8
\end{lstlisting}
\subsection{The Picard group}
\label{m2:pic}
The following computes the Gram matrix for the 21 classes $D_{0,0},\dots,
D_{0,10},D_{1,0},\dots,D_{1,9}$ using Proposition~\ref{prop:intersection-number}.
\begin{lstlisting}[language=Macaulay2]
M21 = matrix table(21,21,(i,j)->(if i == j then -2 else
  if sigma(matrix{pts_i},matrix{pts_j},genericMatrix(S',1,10))==0 then 1 else 0));
\end{lstlisting}
\subsection{Lemma~\ref{lemma:OD}: the group of isometries}
\label{sage:OD}
The Gram matrix of $H^\perp$ is
\[
\begin{psmallmatrix}
-6& -2& -4& -2& -3& -3& -3& -4& -2& -4& -3& -3& -4& -2& -3& -3& -3& -3& -3& -4\\
-2& -4& -2& -2& -1& -2& -2& -3& -2& -2& -1& -2& -3& -2& -1& -2& -2& -2& -2& -3\\
-4& -2& -6& -2& -3& -2& -3& -4& -3& -4& -3& -2& -4& -3& -3& -2& -3& -3& -3& -4\\
-2& -2& -2& -4& -1& -2& -1& -3& -2& -3& -2& -2& -2& -2& -2& -2& -1& -2& -2& -3\\
-3& -1& -3& -1& -4& -1& -2& -2& -2& -3& -2& -2& -3& -1& -2& -2& -2& -1& -2& -3\\
-3& -2& -2& -2& -1& -4& -1& -3& -1& -3& -2& -2& -3& -2& -1& -2& -2& -2& -1& -3\\
-3& -2& -3& -1& -2& -1& -4& -2& -2& -2& -2& -2& -3& -2& -2& -1& -2& -2& -2& -2\\
-4& -3& -4& -3& -2& -3& -2& -6& -2& -4& -3& -3& -4& -3& -3& -3& -2& -3& -3& -4\\
-2& -2& -3& -2& -2& -1& -2& -2& -4& -2& -1& -2& -3& -2& -2& -2& -2& -1& -2& -3\\
-4& -2& -4& -3& -3& -3& -2& -4& -2& -6& -3& -2& -4& -3& -3& -3& -3& -3& -2& -4\\
-3& -1& -3& -2& -2& -2& -2& -3& -1& -3& -4& -2& -2& -2& -2& -1& -1& -2& -2& -2\\
-3& -2& -2& -2& -2& -2& -2& -3& -2& -2& -2& -4& -3& -1& -2& -2& -1& -1& -2& -3\\
-4& -3& -4& -2& -3& -3& -3& -4& -3& -4& -2& -3& -6& -3& -2& -3& -3& -2& -2& -4\\
-2& -2& -3& -2& -1& -2& -2& -3& -2& -3& -2& -1& -3& -4& -2& -1& -2& -2& -1& -2\\
-3& -1& -3& -2& -2& -1& -2& -3& -2& -3& -2& -2& -2& -2& -4& -2& -1& -2& -2& -2\\
-3& -2& -2& -2& -2& -2& -1& -3& -2& -3& -1& -2& -3& -1& -2& -4& -2& -1& -2& -3\\
-3& -2& -3& -1& -2& -2& -2& -2& -2& -3& -1& -1& -3& -2& -1& -2& -4& -2& -1& -3\\
-3& -2& -3& -2& -1& -2& -2& -3& -1& -3& -2& -1& -2& -2& -2& -1& -2& -4& -2& -2\\
-3& -2& -3& -2& -2& -1& -2& -3& -2& -2& -2& -2& -2& -1& -2& -2& -1& -2& -4& -3\\
-4& -3& -4& -3& -3& -3& -2& -4& -3& -4& -2& -3& -4& -2& -2& -3& -3& -2& -3& -6
\end{psmallmatrix}
\]
which has determinant~121. The following verifies that the group of isometries
$\tilde\O(H^\perp)$ is isomorphic to $\GG$.
The variable {\tt M} in the code stands for the above Gram matrix.

\begin{lstlisting}[language=Python]
L = IntegralLattice(M)
OL = L.automorphisms()
D = L.discriminant_group()
OD = D.orthogonal_group()
a, b = L.dual_lattice().gens()[0:2]
u, v = [4*a+9*b, 9*a+8*b]
r = L.dual_lattice().hom(D)
i = OL.Hom(OD)([matrix((r(u*g), r(v*g))) for g in OL.gens()])
G = i.kernel()
print(G.structure_description()) # PSL(2,11)
\end{lstlisting}

One may then proceed to compute the character of this $\GG$-representation and
finds that $(H^\perp)_\bC=(V_{10}')^{\oplus2}$ (which is, in fact, defined over
$\bQ$).

\begin{lstlisting}[language=Python]
ch = G.character(matrix(x).trace() for x in G.conjugacy_classes_representatives())
[(m,c.values()) for (m,c) in ch.decompose()] # [(2, [10, 2, 1, 0, 0, -1, -1, -1])]
\end{lstlisting}

\section{Representation theory of \texorpdfstring{$\GG$}{G}}
\label{sec:appendix}
Below is the character table for the irreducible complex representations of
$\GG$, which can be easily computed in {\tt Sage} (one can also refer to
\cite{atlas}).
\begin{table}[ht]
\scalebox{.8}{
\def\arraystretch{1.5}
$\begin{array}{|r|cccccccc|}
\hline
\text{Conj.\ class}&[\Id]
&\left[\begin{psmallmatrix}1&1\\0&1\end{psmallmatrix}\right]
&\left[\begin{psmallmatrix}1&2\\0&1\end{psmallmatrix}\right]
&\left[\begin{psmallmatrix}0&-1\\1&0\end{psmallmatrix}\right]
&\left[\begin{psmallmatrix}0&1\\-1&-1\end{psmallmatrix}\right]
&\left[\begin{psmallmatrix}2&-2\\2&4\end{psmallmatrix}\right]
&\left[\begin{psmallmatrix}4&0\\0&3\end{psmallmatrix}\right]
&\left[\begin{psmallmatrix}5&0\\0&9\end{psmallmatrix}\right]\\
\text{Size}& 1   & 60 & 60 & 55 & 110 & 110 & 132 & 132             \\
\hline
\bC    & 1   & 1  & 1  & 1  & 1   & 1   & 1               & 1               \\
V_5    & 5  &\frac12\sqrt{-11}-\frac12&-\frac12\sqrt{-11}-\frac12&1&-1&1&0&0\\
V_5^\vee & 5  &-\frac12\sqrt{-11}-\frac12&\frac12\sqrt{-11}-\frac12&1&-1&1&0&0\\
V_{10}  & 10  & -1 & -1 & -2 & 1   & 1   & 0               & 0               \\
V_{10}' & 10  & -1 & -1 & 2  & 1   & -1  & 0               & 0               \\
V_{11}  & 11  & 0  & 0  & -1 & -1  & -1  & 1               & 1               \\
V_{12}  & 12  & 1  & 1  & 0  & 0   & 0   & \frac12\sqrt5-\frac12 &-\frac12\sqrt5-\frac12 \\
V_{12}' & 12  & 1  & 1  & 0  & 0   & 0   &-\frac12\sqrt5-\frac12 & \frac12\sqrt5-\frac12 \\
\hline
\bw3V_{10}^\vee&120&-1&-1&8&3&-1&0&0\\
\hline
\end{array}$}
\caption{\label{tbl:PSL-char}Character table of $\GG=\PSL(2,\bF_{11})$}
\end{table}


\end{document}